\documentclass[12pt, reqno]{amsart}
\usepackage{amsmath,amssymb,amsthm,graphicx,mathrsfs,url} 
\usepackage{amsmath,amssymb,amsthm,tikz}
\usetikzlibrary{decorations.pathreplacing}
\usepackage[toc,page]{appendix}
\usepackage[hidelinks]{hyperref}

\usepackage{comment}

\usepackage[margin=0.9in]{geometry}
\usepackage{hyperref}
\usepackage{comment}
\newcommand{\R}{\mathbb R}

\newtheorem{theorem}{Theorem}[section]
\newtheorem{lemma}[theorem]{Lemma}
\newtheorem{proposition}[theorem]{Proposition}
\newtheorem{definition}[theorem]{Definition}

\newtheorem*{conjecture}{Conjecture}

\theoremstyle{remark}
\newtheorem{remark}[theorem]{Remark}

\newcommand{\eps}{\varepsilon}
\newcommand{\la}{\lambda}

\newcommand{\tr}{\mathrm{tr}}

\numberwithin{equation}{section}

\title[Marked Poincar{\'e} rigidity near hyperbolic metrics]
{Marked Poincar{\'e} rigidity near hyperbolic metrics}
\author{Karen Butt, Alena Erchenko, Tristan Humbert,\\ Thibault Lefeuvre, Amie Wilkinson}
\date{}

\begin{document}

\begin{abstract}
Let $M$ be a compact manifold without boundary equipped with a Riemannian metric $g$ of negative curvature.
In this paper, we introduce the marked Poincar{\'e} determinant (MPD), a homothety invariant of $g$ depending on differentiable periodic data of its geodesic flow.
The MPD associates to each free homotopy class of closed curves in $M$ a number which measures the unstable volume expansion of the geodesic flow along the associated closed geodesic. 
We prove a local rigidity result: if $g$ is sufficiently close to a hyperbolic metric $g_0$ and both metrics have the same MPD and the same  normalized total scalar curvature, then they are homothetic. We also prove that $g_0$ is locally determined by the MPD alone when its Lichnerowicz Laplacian is injective on trace-free divergence-free symmetric 2-tensors.
\end{abstract}
\maketitle

\section{Introduction}

In this paper, we introduce an invariant of a closed, negatively curved manifold that measures the unstable volume expansion of its geodesic flow around periodic orbits.  This invariant, the {\em marked  Poincar\'e determinant},
can be viewed as a first-order variant of the marked length spectrum, as it depends on differentiable data of the flow.
We prove several rigidity results for this invariant on hyperbolic manifolds, including a local analogue of Hamenst\"adt's marked length rigidity for hyperbolic manifolds \cite{ham99}, 
under the additional hypothesis that both metrics have the same \emph{normalized total scalar curvature}. 

We now state our results.
Throughout this paper, we assume $M$ is a $C^\infty$ closed orientable manifold 
of dimension $n \geq 2$, and $g$ is a $C^{\infty}$ Riemannian metric on $M$ of negative sectional curvature { (which we will from now on refer to as a \emph{negatively curved} metric)}. We let $(\phi^t)_{t \in \R}$ denote the geodesic flow of $(M, g)$.

\subsection{Marked Poincar{\'e} rigidity}
We first define the marked Poincar\'e determinant.
Let $SM$ denote the unit tangent bundle of $(M, g)$ and let $X$ be the vector field on $SM$ which generates $(\phi^t)_{t\in \mathbb R}$.
Suppose $v\in SM$ is tangent to a closed geodesic $\gamma$, and let $T = \ell_g(\gamma)$ be the period of $v$. 
Consider a Poincar{\'e} section of $\gamma$ tangent to $X(v)^{\perp} \subset T_v SM$, where the orthogonal complement is taken with respect to the Sasaki metric.
The associated \emph{linearized Poincar{\'e} map} of $\gamma$ is the map 
\[D_v \phi^T : X (v)^{\perp} \to X(v)^{\perp} \]
obtained by restricting $D_v \phi^T : T_v SM \to T_v SM$. 

Recall that since $(M,g)$ is negatively curved, the flow $(\phi^t)_{t\in \mathbb R}$ is \emph{Anosov}, and 
the Anosov splitting $T(SM) = \R X \oplus E_s \oplus E_u$ is invariant under $D_v \phi^t$. In particular, the stable (resp. unstable) bundle $E_s(v)$ (resp. $E_u(v)$) is invariant under $D_v\phi^T.$

Now let $\mathcal{C}$ denote the set of free homotopy classes of closed curves in $M$.
For a negatively curved metric $g$, every nontrivial free homotopy class $c$ of closed curves in $M$ contains a unique geodesic representative $\gamma_g(c)$.
This allows us to make the following definition:

\begin{definition} 
We define the \emph{marked Poincar{\'e} determinant} $($MPD$)$ to be 
\begin{equation*}
\mathcal P_g:  \mathcal C \to \mathbb R, \quad
c \mapsto \det \big( D_v \phi^{T}\big|_{E_u}\big),
\end{equation*}
where $v = (\gamma_g(c))'(0)$ and $T = \ell_g(\gamma_g(c))$.
\end{definition}

In the language of smooth dynamics, the functional $\log (\mathcal{P}_g)$ measures the integrals of the {\em unstable Jacobian} around periodic orbits (see \eqref{eq:logP-intro}).

\begin{remark}\label{rem:homothety} The MPD is invariant under homothety:
for any constant $c > 0$ and any smooth $\phi \in {\rm Diff}^0(M)$ (the elements of ${\rm Diff}(M)$ homotopic to the identity), we have $\mathcal{P}_g = \mathcal{P}_{c \, \phi^* g}$. 
\end{remark}
In order to state our first result, we let 
$$\mathcal S(g):=\mathrm{Vol}_g(M)^{(2-n)/n}\int_M \mathrm{Scal}_{g}\,d\mathrm{vol}_g, $$
denote the \emph{normalized total scalar curvature} of the metric $g$ where $d\mathrm{vol}_g$ is the Riemannian volume form of $g$, $\mathrm{Vol}_g(M)$ is the total volume of $M$, and $\mathrm{Scal}_g$ is the scalar curvature; see Section~\ref{sec:curv} for more details. Note that $\mathcal S$ is also invariant under homothety. Our main result states that the MPD, together with the normalized total scalar curvature, characterize hyperbolic metrics up to rescaling in a neighborhood.
\begin{theorem}
\label{theoS}
Let $(M,g_0)$ be a closed hyperbolic manifold { of dimension $n\geq 3$}. Then there is  $N\in \mathbb N$ and $\epsilon>0$ such that for any smooth metric $g$ with $\|g -g_0\|_{C^N}<\epsilon$, 
one has 
$$\mathcal P_{g}=\mathcal P_{g_0} \quad\text{and}\quad  \mathcal S(g)=\mathcal S(g_0)$$ if and only if there exists a smooth diffeomorphism $\phi \in {\rm Diff}^0(M)$ and $c>0$ so that 
$\phi^* g =  cg_0$. 
\end{theorem}

\begin{remark}
When $n\geq 3$, the normalized total scalar curvature $\mathcal{S}$ always has a saddle point at any hyperbolic metric; see, for instance \cite[Equation (4.2)]{Berger1970}.
This means that $\mathcal S$ alone never determines a hyperbolic metric locally. 
\end{remark}
To prove the main theorem, we relate the kernel of the derivative $d_{g_0}\mathcal P$ to the kernel of the \emph{Lichnerowicz Laplacian} of $g_0$. This operator also appears in the Hessian of the \emph{Liouville entropy} at $g_0$, see \cite[Proposition 5.1.1]{Fla}. We note that Knieper proved a formula relating the Hessian of the total scalar curvature  and of the Liouville entropy at a hyperbolic metric in \cite{Kn}. 

In order to state our next result, we let $\Delta_L$ denote the Lichnerowicz Laplacian, see Definition \ref{def:lich}. We study the action of $\Delta_L$ on trace-free and divergence-free symmetric 2-tensors.
We note that such tensors are often called \emph{TT tensors} for short, where ``TT" stands for ``traceless and transverse". 
\begin{theorem}
\label{theoMain}
Let $(M,g_0)$ be a closed hyperbolic manifold of dimension $n\geq 2$ and suppose that its Lichnerowicz Laplacian is injective on TT tensors. Then there is  $N\in \mathbb N$ and $\epsilon>0$ such that for any {smooth metric} $g$ with $\|g -g_0\|_{C^N}<\epsilon$, 
one has $\mathcal P_{g}=\mathcal P_{g_0}$ if and only if there exists { a smooth diffeomorphism} $\phi \in {\rm Diff}^0(M)$ together with a constant $c > 0$ so that 
$\phi^* g = c g_0$. 
\end{theorem}

\begin{remark}\label{rem:dim2}
    For hyperbolic surfaces, we have $\Delta_L S = -2 S$ for $S$ a TT tensor. 
    In particular, $\Delta_L $ is injective on TT tensors. We also note that a global rigidity holds among Anosov metrics on a closed surface, see Section~\ref{dim2} for more details. 
\end{remark}

\begin{remark}
    From our proof one can also obtain a \emph{stability estimate} on $\mathcal P_g$, similar to the
stability estimates for the marked length spectrum in \cite{GL19,GuKnLef}.
\end{remark}

In positive curvature, the injectivity of $\Delta_L$ follows from Bochner/Weitzenb\"ock formulas, which imply $\Delta_L$ is positive, see for instance \cite{Sch,MST,RST} and the references therein. 
For non-compact hyperbolic manifolds,  positivity of $\Delta_L$ also holds when $n > 9$, as shown by Delay \cite{Del}.

To the best of our knowledge, there is no general statement showing the injectivity result on $\Delta_L$ for {\em compact} negatively curved manifolds. The main difficulty is the fact that  for compact hyperbolic manifolds, the operator $\Delta_L$ acting on TT tensors is  not always positive. 
   It was shown by Flaminio \cite[Theorem C]{Fla} for $n = 3$, and by Maubon \cite{Maubon} for $n \geq 3$, that there exist closed hyperbolic manifolds $(M^n, g_0)$ {of dimension $n$} for which $-n$ is in the spectrum of $\Delta_L$. 
What is more, \cite[Proof of Theorem C]{Fla} shows that the values $\lambda \in \R$ for which there exists a closed hyperbolic $3$-manifold with $\lambda$ in the spectrum of $\Delta_L$ on TT tensors are dense in $[-3,+\infty)$.
   Thus, the injectivity statement cannot be proven from Bochner identities.

Nevertheless, Bonifacio, Mazáč, and Pal \cite{BMP25} give numerical evidence of the existence of hyperbolic $3$-manifolds whose  Lichnerowicz Laplacians are injective on TT tensors. We explain how this follows from their work in Section \ref{sec:examples}.

Finally, we show a local rigidity statement in the conformal class of hyperbolic metrics in any dimension, with no assumption on the Lichnerowicz Laplacian or the scalar curvature.
\begin{theorem}
\label{theo-conf}
Let $(M,g_0)$ be a closed hyperbolic manifold { of dimension $n\geq 2$}. Then there is $N\in \mathbb N$ and $\epsilon>0$ such that for any {smooth negatively curved metric} $g$ conformally equivalent to $g_0$ with $\|g-g_0\|_{C^N}<\epsilon$, 
one has $\mathcal P_{g}=\mathcal P_{g_0}$ if and only if 
there exists 
a constant $c > 0$ so that 
$g = c g_0$. 
\end{theorem}

Motivated by these local rigidity statements, we propose the following conjecture.
\begin{conjecture}
Let $(M, g_0)$ be a negatively curved manifold and suppose $g$ is another negatively curved metric on $M$ with  $\mathcal P_{g_0}=\mathcal P_{g}$.
Then there exists $\phi \in {\rm Diff}^0(M)$ together with a constant $c > 0$ so that 
$\phi^* g = c g_0$. 
\end{conjecture}

In Sections \ref{ss:mls}, \ref{dim2}, and \ref{subsec:lyap}, we give further context and motivation for our above results.

\subsection{Marked length spectrum rigidity}\label{ss:mls}
Our main results can be viewed as dynamically flavored variants of a marked length spectrum rigidity result, more specifically, the local rigidity result of Guillarmou--Lefeuvre \cite{GL19}.
The \emph{marked length spectrum} $\mathcal{L}_g$ of $(M, g)$ is the function which associates to each free homotopy class the length of its unique geodesic representative:
\begin{equation*}\label{eq:MLS}
\mathcal L_g:  \mathcal C \to \mathbb R, \quad
c \mapsto \ell _{g}(\gamma_g (c)).
\end{equation*}

\begin{remark}\label{rem:MLS}
    If $(M, g_0)$ is a real hyperbolic (normalized to have constant sectional curvature $-1$) manifold of dimension $n\geq 2$, then $\log \mathcal{P}_{g_0} = (n-1) \mathcal{L}_{g_0}$ (see Lemma \ref{lemm:Poincare}). 
   For other locally symmetric metrics, $\log \mathcal{P}_{g_0}$ and $\mathcal{L}_{g_0}$ also agree up to a multiplicative constant (see \cite[p.~347]{Ka} for the values of the constants),  
    but outside of the locally symmetric cases, neither of these two functionals determines the other.
\end{remark}

It is conjectured that whenever $g$ and $g_0$ are negatively curved (and more generally Anosov) with $\mathcal{L}_g = \mathcal{L}_{g_0}$, then $g$ and $g_0$ are isometric (more specifically, there exists $\phi \in {\rm Diff}^0(M)$ such that $\phi^* g = g_0$) \cite[Conjecture 3.1]{burnskatok}. 
In other words, the mapping $g \mapsto \mathcal{L}_g$ is (globally) injective on the space of isometry classes of negatively curved metrics on $M$.

Prior to the formulation of this global marked length spectrum rigidity conjecture, Guillemin--Kazhdan considered a related problem, motivated by considerations on the spectrum of the Laplacian. They proved the following 
\emph{deformation rigidity} result: if  $(M, g_0)$ is a closed negatively curved surface and there is a smooth one-parameter family $(g_s)_{s \in (- \eps, \eps)}$ with $\mathcal{L}_{g_s} = \mathcal{L}_{g_0}$ for all $s$, then there is a smooth family $\phi_s \in {\rm Diff}^0(M)$ with $\phi_s^* g_s = g_0$ \cite{GK80}.
We emphasize that Guillemin and Kazhdan's deformation rigidity, more specifically, the smoothness of the family $s \mapsto \phi_s$, is not implied by global rigidity. Their proof additionally establishes the injectivity of the \emph{linearization} $d_g \mathcal{L}$, also known as the \emph{geodesic X-ray transform}, on the space of (divergence-free) symmetric 2-tensors.

In dimension 2, global marked length spectrum rigidity was resolved by Otal and Croke independently \cite{otal, croke} for the negatively curved case (see also \cite{GLP} for the more recent extension to the Anosov case).  
In higher dimensions, the conjecture was shown to hold if $g$ is conformally equivalent to $g_0$ by Katok \cite{ka-conf}, and if $g_0$ is locally symmetric by Hamenst{\"a}dt, 
leveraging the minimal entropy rigidity theorem of Besson--Courtois--Gallot in dimension at least 3 \cite{BCG}. 
More recently,
Guillarmou--Lefeuvre \cite{GL19} (see also subsequent work of Guillarmou--Knieper--Lefeuvre \cite{GuKnLef}) proved the following  \emph{local rigidity} result:
for any negatively curved $(M^n, g_0)$, there exists $N = N(n) \in \mathbb{N}$ and $\eps = \eps(g_0) > 0$ such that for any negatively curved $g$ with $\Vert g - g_0 \Vert_{C^N} < \eps$, the equality $\mathcal{L}_{g} = \mathcal{L}_{g_0}$ implies $g$ and $g_0$ are isometric. 

Theorems \ref{theoS}, \ref{theoMain} and \ref{theo-conf} are partial analogues of these higher dimensional marked length spectrum rigidity results, replacing $\mathcal{L}_g$ with $\mathcal{P}_g$. 
(Note that in dimension 2, injectivity of $g \mapsto \mathcal{P}_g$ follows from known results, as we will elaborate on in Section \ref{dim2} below.)

\subsubsection*{Weighted marked length spectrum rigidity}

The question of injectivity of $g \mapsto \mathcal{P}_g$ also fits into the framework of  ``weighted" marked length spectrum rigidity proposed by Khalil and Lafont \cite[\S 1.5]{GRH} and studied by Gogolev and Rodriguez Hertz \cite{GRH}. 
The setup is to consider weight functions $f_1 : S^{g_1} M \to \R$ and $f_2 : S^{g_2} M \to \R$ and to suppose that they match on all periodic orbits, i.e.,
\[
\int_{\gamma_{g_1}(c)} f_1 \, d \ell_{\gamma_{g_1}(c)}= \int_{\gamma_{g_2}(c)} f_2 \,d \ell_{\gamma_{g_2}(c)}
\]
for all $c \in \mathcal{C}$.
The question is whether this implies $g_1$ and $g_2$ are homothetic. 

Usual marked length spectrum rigidity corresponds to both weight functions $f_1$ and $f_2$ identically equal to 1. The Poincar{\'e} rigidity question we are considering in this paper corresponds to taking $f_1$ and $f_2$ to be the unstable Jacobians of $g_1$ and $g_2$, respectively.
Indeed,
we have 
\begin{equation}\label{eq:logP-intro}
\log \mathcal{P}_g (c) = \int_{\gamma_g(c)} J^u(v) \, d \ell_{\gamma_g(c)},
\end{equation}
where $J^u(v) : = \frac{d}{dt}|_{t =0} \log (\det (D_v \phi^t |_{E^u})) = \frac{d}{dt}|_{t=0} \det (D_v \phi^t|_{E^u})$ denotes the 
\emph{unstable Jacobian}.
Note that in \cite{GRH}, the Khalil--Lafont conjecture was established for $M$ of dimension $n = 2$ and weight functions of regularity $C^r$ for $r > 2$, whereas our weight function $J^u(v)$ is only $C^{1 + \alpha}$ and $n \geq 3$. 

\subsection{Poincar{\'e} rigidity in dimension 2}\label{dim2}
When $M$ has dimension $n = 2$, 
global injectivity of $g \mapsto \mathcal{P}_g$ on Anosov metrics follows by combining work of de la Llave \cite{dL} and Gogolev--Rodriguez Hertz \cite{GRH-abelian} on smooth rigidity of Anosov flows in dimension 3 
with work of Guillarmou--Lefeuvre--Paternain \cite{GLP} on marked length spectrum rigidity. (Note that when $g$ has constant negative curvature, $ \mathcal{L}_g =\log \mathcal{P}_g$ by Remark \ref{rem:MLS}.)
By \eqref{eq:logP-intro}, the hypothesis $\mathcal{P}_{g_1} = \mathcal{P}_{g_2}$, together with the Liv\v sic theorem, implies that the unstable Jacobians of $\phi^t_{g_1}$ and $\phi^t_{g_2}$ are cohomologous. 
De la Llave proved that if this is the case for two $C^0$-conjugate Anosov flows in dimension 3, then the conjugacy is in fact $C^{\infty}$.

While two geodesic flows are not conjugate unless $\mathcal{L}_{g_1} = \mathcal{L}_{g_2}$, they are always orbit equivalent via, e.g., the Morse correspondence (see for instance \cite{gromov}).  
As was explained to us by Andrey Gogolev, de la Llave's argument can be adapted from conjugacies to orbit equivalences \cite[Appendix A]{GLRH25}; as such $\mathcal{P}_{g_1} = \mathcal{P}_{g_2}$ implies the Morse correspondence is $C^{\infty}$. 
For negatively curved metrics, this implies $g_1$ and $g_2$ are homothetic \cite[Corollary 4.6]{ham99}.
For Anosov metrics, \cite[Theorem 7.1, part 2]{GRH-abelian} gives that the flows $\phi^t_{g_1}$ and $\phi^t_{g_2}$ are $C^{\infty}$ conjugate up to a constant rescaling. It now follows from marked length spectrum rigidity \cite{GLP} that $g_1$ and $g_2$ are homothetic.

\subsection{Rigidity of entropies and of Lyapunov exponents}\label{subsec:lyap}

\subsubsection*{Lyapunov rigidity}
Our main results are closely related to work of Butler on characterizing symmetric spaces by their Lyapunov spectra \cite{butler1, butler2}. 
Note that if $v$ is a periodic point of $\phi^t$ with period $T$ and $V \in T_v(SM)$ is an eigenvector of $D_v \phi^T$ with eigenvalue $\sigma$, then one checks that the \emph{Lyapunov exponent} $\lambda(v, V)$ is given by 
\begin{equation}\label{eq:lyap}
    \lambda(v, V) := \limsup_{t \to \infty} \frac{\log \Vert D_v \phi^t(V) \Vert}{t} = \frac{\log |\sigma|}{T}.
\end{equation}
As such, the functional $\Lambda_g := (\log \mathcal{P}_g)/\mathcal{L}_g$ associates to each free homotopy class $c$ the sum of the positive Lyapunov exponents of the closed geodesic $\gamma_g(c)$.

If $(M^n, g)$ has constant negative curvature, then at every periodic orbit, the $n-1$ positive Lyapunov exponents are all equal. 
Butler proved that, in dimension at least 3, this property characterizes metrics of constant sectional curvature among all negatively curved metrics 
\cite{butler1}.
More generally, if the Lyapunov exponents of $g$ on periodic orbits follow the same ``projective pattern" as those of a negatively curved locally symmetric space $g_0$, then $g$ and $g_0$ are homothetic \cite{butler2}.
This can be viewed as another analogue of Hamenst{\"a}dt's aforementioned result characterizing such metrics by the lengths of their closed geodesics. (As in \cite{ham99}, the $n \geq 3$ hypothesis is used to apply \cite{BCG}.)

\subsubsection*{Entropy rigidity}
The techniques in the present paper are also closely related to the techniques in the work of Flaminio \cite{Fla} and the third-named author \cite{Hum} on \emph{Katok's entropy conjecture} \cite{Ka}.
This conjecture, which remains a major open problem, states that if $(M, g)$ is negatively curved and its topological entropy $h_{\rm top}(g)$ coincides with its Liouville entropy $h_{\rm Liou}(g)$
, then $g$ is locally symmetric. Katok established this in dimension 2 \cite{Ka}, and there are
partial results in higher dimensions due to Flaminio \cite{Fla} and the third-named author \cite{Hum}.

The reason their work is related to ours is that the functional  $\mathcal{P}_g$ is the integral of $J^u$ around periodic orbits \eqref{eq:logP-intro}, while Pesin's formula expresses the Liouville entropy as the integral of the same function over $SM$:
\begin{equation}
\label{eq:hLiou}
    h_{\rm Liou}(g) = \int_{SM} J^u(v) dm_g,
\end{equation} 
where $m_g$ is the Liouville measure.

Thus, the coincidence of the functionals $\Lambda_g$ implies  coincidence of the Liouville entropies. On the other hand, we note that this is not implied by coincidence of the $\mathcal{P}_g$ functionals.

\subsubsection*{Local injectivity of $\Lambda_g$}
We note that it follows from Katok's entropy conjecture that $\Lambda_g$ characterizes locally symmetric metrics. 
Indeed, if $\Lambda_g=\Lambda_{g_0}$ for some locally symmetric metric $g_0$, then 
$\Lambda_g$ is a constant, and, hence, $J^u_g(v)$ is cohomologous to a constant, which in turn implies $h_{\rm top}(g) = h_{\rm Liou}(g)$. 
We deduce that $\Lambda_g$ determines any real or complex hyperbolic metric in a small enough neighborhood \cite{Fla, Hum}, and that it characterizes any locally symmetric metric  (globally) in its conformal class \cite{Ka}. 

It also follows from the methods of the present paper and \cite{Fla,Hum} that $d_{g_0}\Lambda$  is injective on TT tensors for $g_0$ being a real or complex hyperbolic metric. 

As emphasized above, the local rigidity of $\mathcal{P}_g$ is much more difficult to obtain than that of $\Lambda_g$ because the differential operator appearing in the derivative $d_{g_0} \mathcal{P}$ is \emph{not} positive, whereas for $\Lambda_g$, the operator appearing in $d_{g_0} \Lambda$ is computed in \cite[Proposition 5.1.1]{Fla} is positive on TT tensors; see \cite{Fla,Hum}.

\subsection{Strategy of the proof}\label{sub:strat}

\subsubsection*{Microlocal techniques}
At a high level, the scheme of the proof of our main theorem is similar to Guillarmou and the fourth-named author's proof of local marked length spectrum rigidity in \cite{GL19}. 
We consider a Taylor expansion of $\mathcal{P}_g$ (more precisely, a closely related functional we will denote by $\Phi_g$) at $g = g_0$ using properties of the \emph{generalized X-Ray transform} introduced by Guillarmou in \cite{Gu}.
In order to apply this machinery, we must study the {solenoidal injectivity} (i.e., the injectivity on divergence-free symmetric 2-tensors) of the derivative $d_{g_0} \Phi$. 

We remark that if $s \mapsto g_s$ is a family of metrics such that $\mathcal{P}_{g_s} \equiv \mathcal{P}_{g_0}$ for all $s$, then $d_{g_0} \Phi \left( \partial_s|_{s=0} g_s \right) = 0$, so solenoidal injectivity of $d_{g_0}\Phi$ is closely related to deformation rigidity.
In simple terms, the generalized X-ray transform is the key to upgrading deformation rigidity to local rigidity. A similar scheme was also used by the third-named author in the context of entropy rigidity \cite{Hum}.

\subsubsection*{Injectivity}
While the derivative of $\mathcal{L}_g$ is easily computed at any metric $g$ using the fact that geodesics minimize length in their free homotopy class, the functional $\mathcal{P}_g$ is more difficult to understand due to the presence of the unstable Jacobian. A standard computation shows that the unstable Jacobian is given by ${\rm tr}(U_g)$, where $U_g$ is the second fundamental form of horospheres (Lemma \ref{lemm:Poincare}). 
Using that $U_g$ is a solution of the Riccati equation, we use work of Flaminio \cite{Fla} to simplify the derivative of $U_g$ at a hyperbolic metric. 
The injectivity of $d_{g_0}\mathcal{P}$ 
reduces to the {solenoidal injectivity} (on tensors with zero mean) of $\Delta_L$.  
Using the fact that $\Delta_L$ preserves TT tensors (see Lemma \ref{lemm:DFG}), 
we reduce the solenoidal injectivity of $\Delta_L$ to the injectivity of $\Delta_L$ on TT tensors. In general, we show that the possible obstruction to  the injectivity of $d_{g_0} \mathcal{P}$ given by the kernel of $\Delta_L$ can be ruled out if we suppose in addition that the normalized total scalar curvatures of $g$ and $g_0$ are equal. This is because the Hessian of normalized total scalar curvature at $g_0$ is non-zero in directions contained in the kernel of the Lichnerowicz Laplacian (see Lemma \ref{lemm:2}).

\subsection*{Organization of the paper}
In Section \ref{sec:prelim}, we recall some properties of symmetric 2-tensors, i.e., tangents to the deformations $s \mapsto g_s$.
In Section \ref{sec:deriv}, we compute the linearization of $\Phi_g$. 
In Section \ref{sec:inj}, we reduce the local injectivity of the MPD to the injectivity of Lichnerowicz Laplacian on TT tensors.
In Section \ref{sec:local}, we prove Theorems \ref{theoS}, \ref{theoMain} and \ref{theo-conf} using a Taylor expansion of $\Phi_g$ and the generalized X-ray transform. In Section \ref{sec:examples}, we explain how to obtain from \cite{BMP25} numerical evidence of the existence of hyperbolic $3$-manifolds whose Lichnerowicz Laplacian is injective on TT tensors, i.e., examples which satisfy the hypotheses of Theorem \ref{theoMain}.

\subsection*{Acknowledgements} We thank Andrey Gogolev for explaining the two-dimensional case (Section \ref{dim2}) to us and Livio Flaminio, Colin Guillarmou, Andrei Moroianu, Uwe Semmelmann and Paul Schwahn for some discussions related to the injectivity of the Lichnerowicz Laplacian.  We further thank Paul Schwahn for noticing a sign mistake in Remark \ref{rem:dim2} in a previous version of the paper.

We also thank Victor Ostrik for pointing us to papers on unitary representations of $\mathrm{SL}(2,\mathbb C)$ relevant to Section \ref{sec:examples} and James Bonifacio for answering some questions on the implementation of the computations done in \cite{BMP25}. We acknowledge the use of LLMs to analyze the list of examples provided in \cite{BMP25aux} to extract the ones that satisfy the condition in Lemma~\ref{lemma: injectivity of Lichnerowicz from the spectrum}. 

Part of this paper was written while K.B., A.E., and T.H. attended an American Institute
of Mathematics (AIM) SQuaRE workshop in September 2025 and the Institute for Advanced Study (IAS) Summer Collaborators Program in June 2026. We would like to thank AIM and IAS
for these opportunities and for their hospitality during our stay.
K.B. was supported by NSF grant DMS-2402173. A.E. was supported by NSF grants DMS-2247230 and DMS-2552860. T.H. and T.L. were supported by the European Research Council (ERC) under the European Union’s Horizon 2020 research and innovation programme (Grant agreement no.\
101162990 — ADG). A.W. was supported by NSF grants DMS--2154796, and DMS--1402852.

\section{Preliminaries}\label{sec:prelim}

Let $(M,g)$ be a closed negatively curved manifold. 
Let $d\mathrm{vol}_{g}$ denote the volume form associated to $g$ and $\mathrm{Vol}_{g}(M)=\int_M d\mathrm{vol}_g$ its total volume.
Let $S^{g}M:=\{(x,v)\in TM\,|\, \|v\|_{g}=1\}$ be its unit tangent bundle. We denote by $(\phi_g^t)_{t\in \mathbb R}$ the geodesic flow generated by $g$ on $S^gM.$

\subsection{Symmetric tensors}
\label{sec:preliSym}
Let $C^{\infty}(M;S^mT^*M)$ be the smooth sections of the bundle of symmetric $m$-tensors on $M$. Note that the scalar product $g$ on $TM$ extends naturally to a scalar product on $C^{\infty}(M;S^mT^*M)$, which we will denote by $\langle \cdot, \cdot\rangle_{L^2(M;S^mT^*M)}$ (or $\langle \cdot,\cdot\rangle$ when there is no risk of confusion). The \emph{trace} is given by
\begin{equation}
\label{eq:trace}
\tr_{g}: C^{\infty}(M;S^{m+2}T^*M)\to C^{\infty}(M;S^{m}T^*M),\quad S\mapsto \sum_{i=1}^nS_x(e_i,e_i,\ldots), 
\end{equation}
 where $(e_i)_{i=1}^n$ is a $g$-orthonormal basis of $T_xM$. Note that for $m=2$, the space of symmetric $2$-tensors splits as
\begin{equation}
\label{eq:decomp1}C^{\infty}(M;S^2T^*M)=C^{\infty}(M;S^2_0T^*M)\oplus C^{\infty}(M)g, 
\end{equation}
where 
$
C^{\infty}(M;S^2_0T^*M):=\{S\in C^{\infty}(M;S^2T^*M)\mid \tr_{g}(S)=0\}
$
denotes the bundle of trace-free tensors. 
We will frequently identify symmetric tensors and functions on the unit tangent bundle $SM$ as follows. Let
\begin{equation}
\label{eq:pim} \pi_m^*:C^{\infty}(M;S^mT^*M)\to C^{\infty}(SM), \quad \pi_m^*S(x,v)=S_x(\underbrace{v,\ldots,v}_{m \text{ times.}}).
\end{equation}

The Levi-Civita connection $\nabla_g$ acts naturally on $m$-tensors, but it does not preserve the symmetry. We thus introduce the \emph{symmetrized covariant derivative}: 
$$D_{g}:=\mathrm{Sym}\circ \nabla_{g}: C^{\infty}(M;S^mT^*M)\to C^{\infty}(M;S^{m+1}T^*M).$$
We note the following relation between the symmetrized covariant derivative and the generator of the geodesic flow (see \cite[Lemma 14.1.9]{Lef}):
\begin{equation}
    \label{eq:XandD}
    X_g\pi_m^*=\pi_{m+1}^*D_g.
\end{equation}
The formal adjoint of $D_g$ is the \emph{divergence} operator $D_{g}^*$:
$$D_{g}^*=-\tr_g \circ \nabla_{g}: C^{\infty}(M;S^mT^*M)\to C^{\infty}(M;S^{m-1}T^*M). $$
The \emph{rough Laplacian} is given by
\begin{equation}
    \nabla^*_g\nabla_g: C^{\infty}(M;S^mT^*M)\to C^{\infty}(M;S^{m}T^*M).
\end{equation}
When there is no risk of confusion on the metric, we will suppress the $g$ subscripts.

\subsection{Decomposition of the space of symmetric tensors.}
 There exists a natural gauge given by the action of the group $\mathrm{Diff}^0(M)$ of smooth diffeomorphisms homotopic to the identity. We define
$$\mathcal O(g):=\{\phi^*g\mid \phi \in \mathrm{Diff}^0(M)\},\quad T_{g}\mathcal O(g)=\{\mathcal L_Vg\mid V\in C^{\infty}(M;TM)\}. $$ 
We remark the following fact:
\begin{equation}
\label{eq:tangentspace}T_{g}\mathcal O(g)=\{D_{g}p\mid p\in C^{\infty}(M;TM)\}.
\end{equation}
Tensors of the form $D_g p$ are called \emph{potential tensors}.
A natural transverse slice to $T_g \mathcal{O}(g)$ is provided by the kernel of the adjoint $D_{g}^*$ \cite[Theorem 14.1.10]{Lef}. Elements of $C^{\infty}(M;S^mT^*M)\cap \mathrm{Ker}(D_{g}^*)$ are called \emph{divergence-free} (or \emph{solenoidal}) tensors.
For any $S\in C^{\infty}(M;S^mT^*M)$, there exists a  pair 
\begin{equation}
\label{eq:decomppot}(p,h)\in C^{\infty}(M;S^{m-1}T^*M)\times \big(C^{\infty}(M;S^mT^*M)\cap \mathrm{Ker}(D_{g}^*)\big),\quad S=D_{g}p+h.
\end{equation}
Let $(g_{\la})_{\la \in (-\eps, \eps)}$ be a 
    deformation of a metric $g_0$ which satisfies $\mathcal{P}_{g_{\la}} \equiv \mathcal{P}_{g_0}$ and let $S$ denote $\partial_{\la}|_{\la = 0} g_{\la}$.
Since $\mathcal{P}_g$ is unchanged by isometries and rescalings of $g$, we see that $d_{g_0}\mathcal{P}[S] = 0$ for $S$ a potential tensor (see \eqref{eq:tangentspace}) or $S$ of the form $cg_0$ for some constant $c$. As such, we now study the kernel of $d_{g_0}\mathcal{P}$ restricted to the ``transverse slice" 
 \begin{equation}
    \label{eq:divzeromean}
    \mathcal V_{g_0}:=\mathrm{Ker}(D_{g_0}^*)\cap 
    \{cg_0\mid c\in \mathbb R\}^\perp.
\end{equation}
Elements of $\mathcal{V}_{g_0}$ are the divergence-free symmetric 2-tensors tangent to volume-preserving deformations of $g_0$.

We have the following lemma which allows one to ``project" a metric $g$ onto $\mathcal V_{g_0}$. 
\begin{lemma}[Slice lemma]
\label{slice lemma}
Let $k\geq 2$,  and $\alpha\in (0,1)$. Then there exists a neighborhood $\mathcal U$ of $g$ in the $C^{k,\alpha}$-topology such that for any $g'\in \mathcal U$, there is a unique $\phi_{g '}\in \mathrm{Diff}^0(M)$ of regularity $C^{k+1,\alpha}$, close to the identity, together with a constant $c_{g'}>0$ such that $c_{g'}\phi_{g'}^*g'-g\in \mathcal V_{g}$. Moreover, there exists $\epsilon>0$ and $C>0$ such that
$$\|g'-g\|_{C^{k,\alpha}}\leq \epsilon \ \Rightarrow \ \|c_{g'}\phi_{g'}^*g'-g\|_{C^{k,\alpha}}\leq C\|g'-g\|_{C^{k,\alpha}}.$$
\end{lemma}
\begin{proof}
    The lemma was essentially obtained in this form in \cite[Lemma 2.4]{GuKnLef}, but the idea goes back to Ebin \cite{Eb}. By \cite[Lemma 2.4]{GuKnLef}, one can find a neighborhood $\mathcal U$ of $g$ in the $C^{k,\alpha}$-topology such that for any $g'\in \mathcal U$, there is a unique $\phi_{g '}\in \mathrm{Diff}^0(M)$ of regularity $C^{k+1,\alpha}$, close to identity such that $\phi_{g'}^*g'\in \mathrm{Ker}(D_{g}^*)$ is divergence-free. Moreover, there exists $\epsilon>0$ and $C>0$ such that
\begin{equation}
\label{eq:slice}\|g'-g\|_{C^{k,\alpha}}\leq \epsilon \ \Rightarrow \ \|\phi_{g'}^*g'-g\|_{C^{k,\alpha}}\leq C\|g'-g\|_{C^{k,\alpha}}.
\end{equation}
We choose a constant $c_{g'}>0$ such that $c_{g'}\phi_{g'}^*g'-g$ has zero-mean. This is equivalent to $$\int_{M}\tr_{g}(c_{g'}\phi_{g'}^*g'-g)\,d\mathrm{vol}_{g}=0\iff c_{g'}^{-1}=\frac{1}{n\mathrm{Vol}(M)}\int_M \tr_{g}(\phi_{g'}^*g')\,d\mathrm{vol}_{g}. $$
Note that we have $D_g^*(c_{g'}\phi_{g'}^*g'-g)=0$ and hence $c_{g'}\phi_{g'}^*g'-g\in \mathcal V_{g_0}$.
Moreover, the mapping $g'\mapsto c_{g'}$ is smooth and, more precisely, we have 
\begin{equation}\label{eq:c_g}
    \|g'-g\|_{C^{k,\alpha}}\leq \epsilon \ \Rightarrow \ |c_{g'}-1|\leq C\|g'-g\|_{C^{k,\alpha}}, 
\end{equation}
for some constant $C>0$ independent of $g'.$ Combined with \eqref{eq:slice}, this gives the statement of the lemma.
\end{proof}
\subsection{Curvature tensors}
\label{sec:curv}
The \emph{Ricci tensor} is
\begin{equation*}
    \label{eq:Ric}
    \mathrm{Ric}_g\in C^{\infty}(M;S^2T^*M),\quad \mathrm{Ric}_g(v,w):=\mathrm{tr}_g(y\mapsto \mathbf{R}_g(v,y)w)=\sum_{i=1}^ng(\mathbf{R}_g(v,e_i)w,e_i),
\end{equation*}
for any $g$-orthonormal basis $(e_i)_{i=1}^n$ and where $\mathbf{R}_g$ is the Riemannian curvature tensor of $g$. For $g_0$ hyperbolic, one has $\mathrm{Ric}_{g_0}=-(n-1)g_0.$
The \emph{scalar curvature} is
\begin{equation*}
    \label{eq:Scal}
    \mathrm{Scal}_g\in C^{\infty}(M),\quad \mathrm{Scal}_g(x):=\mathrm{tr}_g(\mathrm{Ric}_g)(x)=\sum_{i=1}^n\mathrm{Ric}_g(e_i,e_i),
\end{equation*}
for any $g$-orthonormal basis $(e_i)_{i=1}^n.$ For $g_0$ hyperbolic, one has $\mathrm{Scal}_{g_0}=-n(n-1).$ The \emph{normalized total scalar curvature} is
\begin{equation*}
    \label{eq:S}
    \mathcal S(g):=\mathrm{Vol}_g(M)^{(2-n)/n}\int_M \mathrm{Scal}_{g}d\mathrm{vol}_g.
\end{equation*}

\begin{definition}\label{def:lich}
   The \emph{Lichnerowicz Laplacian} is given by
\begin{equation}\label{eq:lich}
(\Delta_L)_gS:=\nabla^*_g\nabla_gS+\mathrm{Ric}_g\circ S+S\circ \mathrm{Ric}_g-2R_g^\circ(S),
\end{equation}
where, for any $g$-orthonormal basis $(e_i)_{i=1}^n$,
\begin{equation*}
\label{eq:Ricci} R_g^{\circ}(S)(X,Y)=-\sum_{i=1}^n S(\mathbf{R}_g(e_i,X)Y,e_i),\quad S\circ \mathrm{Ric}_g(X,Y)=\sum_{i=1}^n S(\mathbf{R}_g(e_i,X)e_i, Y).
\end{equation*}
\end{definition}
Note that when $g=g_0$ is a hyperbolic metric, one has
\begin{equation}
    \label{eq:LaplFla}
    \Delta_LS=\nabla^*\nabla S-2nS+2\tr(S)g_0,
\end{equation}
see \cite[Proof of Proposition 1.3.3]{Fla}.

\section{Linearization of the Poincaré determinant}\label{sec:deriv}
In this section, we compute the first derivative of the Poincaré determinant at a hyperbolic metric $g_0.$ We start by expressing the Poincaré determinant using the positive solution to the Riccati equation $U_g$, or equivalently, the second fundamental form of horospheres. For a closed geodesic $\gamma$, we denote by $d\ell_{\gamma}$ the (non-normalized) one-dimensional Lebesgue measure supported on~$\gamma.$
\begin{lemma}
\label{lemm:Poincare}
Let $(M^n,g)$ be a smooth closed negatively curved manifold. Let $\mathcal{D}_g = \log \mathcal{P}_g$. Then 
    \begin{equation*}
        \label{eq:trUg}
        \forall c\in \mathcal C, \quad  \mathcal D_g(c)= \int_{\gamma_g(c)}\tr(U_g)d\ell_{\gamma_g(c)} .
    \end{equation*}
In particular, for a hyperbolic metric $g_0$, we have $\mathcal{D}_{g_0}(c) = (n-1) \ell_{g_0} (\gamma_{g_0}(c))$.  
\end{lemma}

\begin{proof}
Let $\gamma:=\gamma_g(c)$ and let $v = \gamma'(0) \in S_xM$.
For a Jacobi field $J(t)$ along $\gamma(t)$, we will use the notation $J'(t)$ for the covariant derivative $\nabla_{\gamma'(t)} J(t)$.
Let $V \in T_v(SM)$ and let $(V_h, V_v) \in T_x M \oplus T_x M$ denote its decomposition into horizontal and vertical components.
Then $$D \phi^t (V) = (J(t), J'(t)),$$ where $J$ is the Jacobi field along $\gamma$ with initial conditions $(J(0), J'(0)) = (V_h, V_v)$, see for instance \cite[Proposition 1.13]{Bal}.

Let $U(v)$ denote the second fundamental form of the unstable horosphere determined by $v$.
We make the identification
\begin{align}\label{ident}
    \mathcal{N}(v) \to E^u(v), \quad w \mapsto (w, U(v) w) \in E^u(v),
\end{align}
where $\mathcal N(v)$ denotes the normal bundle to $v$ and 
where the right-hand side is understood in terms of the  identification $T_x M \oplus T_x M \cong T_v(SM)$ using horizontal and vertical components.
Using \eqref{ident}, we can view $D \phi^t|_{E^u}$ as a linear map $A(t): \mathcal{N}(v) \to \mathcal{N}(\phi^t v)$.
Whenever $J(t)$ is an unstable Jacobi field along $\gamma(t)$, 
one checks that $J'(t) = U(\phi^t v)(J(t))$.
This means that the map $A(t): \mathcal{N}(v) \to \mathcal{N}(\phi^t v)$ satisfies the equation
\[
\frac{D}{dt} A(t) = U(\phi^t v) A(t), 
\]
where $\frac{D}{dt}$ denotes covariant differentiation along $\gamma(t)$. 
Therefore, 
\[
\frac{d}{dt} \det A(t) = \det A(t) \, {\rm tr}\left( A(t)^{-1} \frac{D}{dt} A(t) \right) \implies \frac{d}{dt} \log \det A(t) = {\rm tr}(U(\phi^t v)).
\]
Integrating the above from $0$ to $T$ completes the proof.
The second statement follows from the fact that for $g_0$ hyperbolic, we have that $U_{g_0}$ is the identity map. 
\end{proof}

In order to take the derivative of the map $\mathcal D$, we first justify that it is smooth in $g$. Fix a negatively curved metric $g_0$ on $M$. 
\begin{lemma}
\label{lemm:Taylor}
 Let $\mathrm{Met}^k_{<0}$ denote the space of $C^k$ negatively curved metrics on $M$.
    For $k\geq 3$, the mapping
    $g\mapsto \mathcal{D}_g:\mathrm{Met}^k_{<0}\to \mathbb R^{\mathcal C}$
    is $C^{k-3}$.
\end{lemma}

\begin{proof}
Let $(M,g_0)$ be a closed negatively curved manifold and let $g$ be a negatively curved metric. 
 Let $H_g:S^{g_0}M\to S^{g}M$ be the $(g,g_0)$-Morse correspondence. Recall that it satisfies
\begin{equation}
\label{eq:Morse}\forall v \in S^{g_0}M,\quad H_g\circ \phi_{g_0}^t(v)=\phi_{g_0}^{\alpha_g(t,v)}\circ H_g(v), 
\end{equation}
where $\alpha_g(t,v)$ is a time change. Let 
\begin{equation}
\label{eq:alpha}\forall v \in S^{g}M,\quad \alpha_g(v):=\frac{d}{dt}|_{t=0}\,\alpha_g(t,v). 
\end{equation}
For a free homotopy class  $c\in \mathcal C$ and for $v\in \gamma_g(c)$, using \eqref{eq:Morse} and \eqref{eq:alpha}, we have 
\begin{align*}\int_{\gamma_g(c)}&\tr(U_g)d\ell_{\gamma_g(c)}=\int_0^{\ell_g(\gamma_g(c))}\tr(U_g)(\phi_g^tv)dt= \int_0^{\ell_g(\gamma_g(c))}\tr(U_g)\circ H_g^{-1}\circ\phi_{g_0}^{\alpha_g(t,v)}\circ H_g(v)dt
\\&=\int_0^{\ell_{g_0}(\gamma_{g_0}(c))}\tr(U_g)\circ H_g^{-1}\circ\phi_{g_0}^{s}\circ H_g(v)\alpha_g(\phi_{g_0}^{s}\circ H_g(v))ds=\int_{\gamma_{g_0}(c)}\alpha_g\cdot \tr(U_g)\circ H_g^{-1}d\ell_{\gamma_{g_0}(c)}.
\end{align*}
The result of the lemma now follows from the fact that 
$$\mathrm{Met}^k_{<0} \to C^{\nu}(SM), \quad g \mapsto {\rm Tr}(U_g)\circ H_g^{-1} \text{ and }g\mapsto \alpha_g, $$
are both $C^{k-3}$ mappings for any $\nu \in (0, 1)$. The first statement was shown in \cite{Con} while the second one was proved in \cite{KKPW}, see also \cite[Theorems 3.3.6 and 3.3.7]{Fla}.
\end{proof}
Now let $g_0$ be a hyperbolic metric on $M$.
For another negatively curved metric $g$ on $M$, we let $\Phi_g:=\mathcal D_g/\mathcal D_{g_0}$. 
The next proposition computes the first derivative of $\Phi$ at $g = g_0.$
\begin{proposition}\label{prop:deriv}
Let $(M^n,g_0)$ be a closed hyperbolic manifold. 
Then for any $S\in C^{\infty}(M;S^2T^*M)$,  the map $d_{g_0}\Phi(S) : \mathcal{C} \to \mathbb R$ is given by
\begin{equation*} 
\forall c\in \mathcal C, \quad  d_{g_0}\Phi(S)(c) = -\dfrac{1}{\mathcal{D}_{g_0}(c)}\int_{\gamma_{g_0}(c)} \frac{1}{4}\pi_2^* \Delta_LS d\ell_{g_0} 
=-\frac{1}{4(n-1)\ell_{g_0}(c)}\int_{\gamma_0(c)}\pi_2^*\Delta_L Sd\ell_{\gamma_0(c)}.
\end{equation*}
\end{proposition}

\begin{proof}
Let $(g_{\lambda})_{\lambda \in(-\epsilon, \epsilon)}$ be a deformation of $g_0$ such that $\partial_\la|_{\la=0}g_\la=S$. Fix a free homotopy class $c\in \mathcal C$. In the following computation, we will write $\gamma_{\lambda}$ instead of $\gamma_{\la}(c)$. Differentiating, we obtain
\begin{align*}
\partial_{\lambda}|_{\lambda = 0}\left( \int_{\gamma_{\lambda}} {\rm tr}(U_{\lambda}) d \ell_{\gamma} \right)&= \partial_{\lambda} |_{\lambda = 0}\left(\int_{\gamma_{\lambda}} (n-1) d \ell_{\la}\right) + \int_{\gamma_0} {\rm tr} (\partial_{\la}|_{\lambda = 0} U_{\la}) d \ell_{0},
\end{align*}
where we used that $\tr(U_0)=n-1$.
For the first term, we have
\begin{align*}\label{eq:Xray}
    \partial_{\lambda} |_{\lambda = 0}\int_{\gamma_{\lambda}} d \ell_{\la} &= \partial_{\la}|_{\lambda = 0} \ell_{g_{\la}} (\gamma_{\lambda}) = \partial_{\la}|_{\lambda = 0}\big( \ell_{g_{0}} (\gamma_{\lambda}) + \ell_{g_{\la}} (\gamma_{0}) \big)= \partial_{\la}|_{\lambda = 0} \int_{\gamma_0} \sqrt{g_{\lambda} (v, v)} = \frac{1}{2} \int_{\gamma_0} \pi_2^* S,
\end{align*}
where we used that $\gamma_0$ minimizes the $g_0$-length in its free homotopy class. 

The second term was computed by Flaminio (see \cite[Corollary 4.3.2]{Fla}
\footnote{Note that in \cite{Fla} $U$ is the \emph{negative} Riccati solution, so our formulas differ by a sign.}), using that $U$ is a solution to the \emph{Riccati equation},
and is equal to 
$$\int_{\gamma_0(c)}\tr(\partial_\la|_{\la=0}{U}_\la )d\ell_{\gamma_0(c)}=-\frac{1}{2}\int_{\gamma_0(c)}\big(\partial_\la|_{\la=0}\mathrm{Ric}_\la +(n-1)\pi_2^*S \big)d\ell_{\gamma_0(c)}. $$
Using \cite[1.174 and 1.180 b)]{Bes}, we have
\begin{equation} 
\label{eq:LichneRicci}\partial_\la|_{\la=0}\mathrm{Ric}_\la =\frac 12 \Delta_L S-D_{g_0}D_{g_0}^*(S)-\frac 12\nabla d(\tr(S)). 
\end{equation}
To complete the proof, we first note that $\nabla d(\tr(S))=D_{g_0} d(\tr(S))$, because $\nabla d(\tr(S))$ is already a symmetric tensor since it is the Hessian of the function $\tr(S).$ Using \eqref{eq:XandD}, this means that 
$$ \pi_2^*\partial_\la|_{\la=0}\mathrm{Ric}_\la =\frac 12 \pi_2^*\Delta_L S-X\big(\pi_1^*D_{g_0}^*(S)\big)-\frac 12X\big(\pi_1^* d(\tr(S))\big). $$
Using the $X$-invariance of the measure $d \ell_{\gamma_0(c)}$ concludes the proof.
\end{proof}

\section{Kernel of the Lichnerowicz Laplacian on TT tensors}\label{sec:inj}
In this section, we solve the \emph{infinitesimal problem} corresponding to Theorem \ref{theoMain}, that is, we relate the kernel of $d_{g_0} \mathcal{P}$ to the kernel of $\Delta_L$ on TT tensors.

\begin{proposition}
\label{prop:41}
    Let $(g_{\la})_{\la \in (-\eps, \eps)}$ be a 
    deformation of a hyperbolic metric $g_0$ which satisfies $\mathcal{P}_{g_{\la}} \equiv \mathcal{P}_{g_0}$. 
    Assume in addition that $S = \partial_{\la}|_{\la = 0} g_{\la} \in \mathcal{V}_{g_0}$. 
    Then $S$ is a TT tensor with $\Delta_L S = 0$. 
\end{proposition}

The main ingredient of the proof is the following observation.
\begin{lemma}
\label{lemm:DFG}
    Let $(M^n, g_0)$ be a hyperbolic manifold. If $S$ is a divergence-free symmetric 2-tensor, then the symmetric 2-tensor $\Delta_L S$ is also divergence-free. 
\end{lemma}

\begin{proof}
    Since $g_0$ is hyperbolic, we have the following additional commutation relation:
\begin{equation}
\label{eq:DFG}
    [\nabla^* \nabla, D^*_{g_0}]S = -(n+1) D_{g_0}^* S - 2 D_{g_0}( {\tr}(S)),
\end{equation}
    see \cite[Equation (C.1)]{DFG}. For $S \in C^{\infty}(M, S^2 T^*M)$ a symmetric 2-tensor, we decompose $S = S_2+S_0 g_0$, where $S_0 \in C^{\infty}(M)$ and $S_2$ trace-free, see \eqref{eq:decomp1}. 
    Using \eqref{eq:LaplFla}, we see that if $S$ is divergence-free, then 
    $\Delta_LS=\nabla^* \nabla S - 2n S_2. $
    This means that, using \eqref{eq:DFG}, 
\begin{align*}
    D_{g_0}^* (\nabla^* \nabla S - 2n S_2) = 2 D_{g_0} ({\tr}(S)) - 2n D_{g_0}^* S_2= 2n D_{g_0}(S_0) - 2n D_{g_0}^* S_2.
\end{align*}
    The last line is zero by the divergence-free condition:
    \begin{equation}
    0 = D_{g_0}^*S = D_{g_0}^* S_2 - {\tr}(\nabla(S_0g_0)) = D_{g_0}^* S_2 -  \nabla S_0,
\end{equation}
which completes the proof. 
\end{proof}

\begin{proof}[Proof of Proposition \ref{prop:41}]
Let $(g_{\la})_{\la \in (-\eps, \eps)}$ be a deformation of a hyperbolic metric $g_0$ satisfying $\mathcal{P}_{g_{\la}} \equiv \mathcal{P}_{g_0}$ for all $\la\in (-\eps,\eps)$ and such that 
$S=\partial_\la|_{\la=0}g_\la\in \mathcal{V}_{g_0}$. 
Differentiating the isospectrality assumption at $\lambda = 0$ and using Proposition \ref{prop:deriv} gives
\begin{equation}\label{eq:livsic}
\int_{\gamma} \pi_2^* \Delta_L S \, d \ell_{\gamma} = 0
\end{equation}
for all closed $g_0$-geodesics $\gamma$. Using the Liv\v sic theorem and the solenoidal injectivity of the X-ray transform on symmetric $2$-tensors \cite{CS98}, this means that $\Delta_L S$ is a potential tensor, equivalently, $S\in \mathrm{Ker}(\Pi_{\mathrm{Ker}(D_{g_0}^*)}\Delta_L)$. 

We will now show that $S$ is trace-free and in the kernel of $\Delta_L$. More precisely, we claim that:
     \begin{equation} 
     \label{prop:DFG}\mathrm{Ker}(\Pi_{\mathrm{Ker}(D_{g_0}^*)}\Delta_L)\cap \mathcal V_{g_0}=\mathrm{Ker}(\Delta_L)\cap \mathcal V_{g_0}=\mathrm{Ker}(\Delta_L)\cap\mathrm{Ker}(\tr)\cap \mathrm{Ker}(D_{g_0}^*). 
     \end{equation}
Indeed, suppose that $S\in \mathrm{Ker}(\Pi_{\mathrm{Ker}(D_{g_0}^*)}\Delta_L) \cap \mathrm{Ker}(D_{g_0}^*).$ On the one hand, $S\in \mathrm{Ker}(\Pi_{\mathrm{Ker}(D_{g_0}^*)}\Delta_L)$ means $\Delta_LS$ is potential, i.e., there exists $p\in C^{\infty}(M;T^*M)$ such that $\Delta_L S=D_{g_0}p$.
On the other hand, $S \in \mathcal{V}_{g_0} \subset {\rm Ker}(D^*_{g_0})$ together with Lemma \ref{lemm:DFG} gives $\Delta_L S$ is divergence-free.
Hence,
$$\Vert \Delta_L S \Vert^2_{L^2(M;S^2T^*M)} = \langle \Delta_L S, D_{g_0}p \rangle_{L^2(M;S^2T^*M)} = \langle D_{g_0}^* \Delta_L S, p \rangle_{L^2(M;T^*M)} = 0,$$ which shows $S \in {\rm Ker}(\Delta_L) \cap \mathcal{V}_{g_0}$ as desired.

To prove the second equality in \eqref{prop:DFG}, write $S = S_0 + \frac{{\rm tr}(S)}{n} g_0$ for $S_0$ trace-free. 
Note from \eqref{eq:LaplFla} that on conformal tensors, one has $\Delta_L(fg_0)=(\Delta f)g_0$.
Hence,
$0 = \Delta_L S = \Delta_L S_0 + \frac{1}{n}(\Delta {\rm tr}(S) )g_0$. Since $\Delta_L$ commutes with the trace, we have that $\Delta_L S_0$ is also trace-free, which means that $\Delta ({\rm tr}(S)) = 0$. By \eqref{eq:divzeromean}, we conclude ${\rm tr}(S) = 0$, which proves that the kernel of $\Delta_L$ on $\mathcal V_{g_0}$ coincides with its kernel on TT tensors.

Since $S\in \mathcal V_{g_0}$, the statement of the Proposition follows from \eqref{prop:DFG}.
\end{proof}

\section{Local rigidity results}\label{sec:local}

In this section, we prove Theorems \ref{theoS}, \ref{theoMain} and \ref{theo-conf}. 
The key tool is the \emph{generalized X-Ray transform} $\Pi$, whose key properties we now briefly recall. For any smooth function $f\in C^{\infty}(SM)$ such that $\int_{SM}f dm_g=0$ (where $m_g$ denotes the Liouville measure of $g$), we define
\begin{equation}
\label{eq:Pi}
\quad \langle \Pi f,f\rangle=\int_{\mathbb R}\langle f\circ \varphi_t,f\rangle_{L^2} \, dt,
\end{equation}
where the integral converges by the exponential decay of correlations \cite{Liv}, see for instance \cite[Equation $(2.6)$]{GuKnLef}. (A microlocal definition of the operator $\Pi$ was originally given in \cite{Gu}.)

We will use the following properties of $\Pi$:
\begin{enumerate}
    \item\label{PiX} One has $\Pi X = 0$, see \cite[Theorem 1.1]{Gu}.
    \item\label{Pi-bdd} The operator ${\pi_2}_* \Pi : C^s(SM) \to C^s(M,S^2 T^*M)$ is bounded for all $s > 0$ not an integer by  \cite[Lemma 5.10]{GBL}, see also \cite[Lemma 16.2.11]{Lef}.
    \item\label{eq:coercivePI} We have the following \emph{coercive estimate} on solenoidal tensors. For any $s\notin \mathbb Z$, there is $C>0$ such that for any $S$ of \emph{zero mean} (i.e. $S$ for which ${\rm tr}( S)$ has mean zero), we have
    \begin{equation*}
   \|\Pi_{\mathrm{Ker}(D_{g_0}^*)}S\|_{C^s}\leq C\| {\pi_2}_* \Pi \pi_2^* S\|_{C^{s+1}},
    \end{equation*}
    where $\Pi_{\mathrm{Ker}(D_{g_0}^*)}$ denotes as before the orthogonal projection onto $\mathrm{Ker}(D_{g_0}^*)$, see \cite[Lemma 5.9]{GBL}.
\end{enumerate}

The operator $ {\pi_2}_*\Pi\pi_2^* $ was used crucially by Guillarmou--Lefeuvre in their proof of the local rigidity of the marked length spectrum \cite{GL, GuKnLef} (and by Humbert for the proof of Katok's entropy conjecture near real and complex hyperbolic metrics \cite{Hum}). The analogue of this operator in our setting is the \emph{generalized Poincar{\'e} X-ray transform} 
\begin{equation}
\label{eq:Q}
Q(S) := {\pi_2}_* \Pi \pi_2^* \Delta_LS.  
\end{equation}
We have a coercive estimate for $Q$ perpendicular to the kernel of $\Delta_L$, in analogy with property \eqref{eq:coercivePI} above.
\begin{lemma}\label{coerciveQ}
    For any $S \in \mathcal{V}_{g_0}$ such that $\Pi_{{\rm Ker} \Delta_L} S = 0$, one has for any $s \notin \mathbb{Z}$
        \begin{equation*}
        \|S\|_{C^s}\leq C\|Q(S)\|_{C^{s-1}}.
    \end{equation*}
\end{lemma}

\begin{proof}
    Since $\Delta_L S$ has zero mean, property \eqref{eq:coercivePI} gives that for any $s \notin \mathbb{Z}$ there exists $C > 0$ such that 
    \[
     \|\Pi_{\mathrm{Ker}(D_{g_0}^*)}\Delta_L S\|_{C^{s-2}}\leq C\|{\pi_2}_* \Pi\pi_2^*\Delta_LS\|_{C^{s-1}} = C \Vert Q(S) \Vert_{C^{s-1}}.
    \]
    Next, \eqref{prop:DFG} implies $\Pi_{\mathrm{ker}(D_{g_0}^*)}\Delta_L S=\Delta_L S$ 
    for all $S\in \mathcal V_{g_0}$. The standard elliptic estimate applied to $\Delta_L$ then gives 
\[
\|S\|_{C^s}^2-\|\Pi_{\mathrm{ker}(\Delta_L)}S\|_{C^s}^2 \leq 
C \| \Delta_L S \|_{C^{s-2}}^2 =
C\|\Pi_{\mathrm{ker}(D_{g_0}^*)}\Delta_LS\|^2_{C^{s-2}}
\]
for all $S \in \mathcal{V}_{g_0}$ and some constant $C>0$, which completes the proof.
\end{proof}

We now use a similar setup to \cite{GL19} to prove Theorems \ref{theoS}, \ref{theoMain} and \ref{theo-conf}, which involves Taylor expanding $\mathcal{P}$ at the hyperbolic metric $g_0$. 
Let $(M^n,g_0)$ be a closed hyperbolic manifold. 
 Let $g$ be a metric such that  $\|g-g_0\|_{C^N}<\epsilon$ for a small $\epsilon $ and a large $N$ to be determined later. Using the slice lemma (Lemma \ref{slice lemma}), we choose a diffeomorphism $\phi_g$ such that $\phi_g^* g - g_0$ is divergence-free; we then choose a constant $c_g>0$ such that 
  \begin{equation}
  \label{eq:projection}S:=c_g\phi_g^*g-g_0\in \mathcal V_{g_0},
  \end{equation}
  see \eqref{eq:divzeromean}.
  
  Suppose that $\mathcal{P}_g = \mathcal{P}_{g_0}$.
  Since $\mathcal{P}$ is invariant by isometries and rescalings of $g$, we have $\Phi_{c_g \phi_g^* g} = 1$ as well. Our goal is to show that $S = 0$, which implies $g$ and $g_0$ are homothetic. 

Without further hypotheses on $\Delta_L$, we have the following control on~$Q(g-g_0)$. Note that we do not yet need to project $g-g_0$ onto $\mathcal V_{g_0}$ as in \eqref{eq:projection} for this step. 
The idea is to use the equality $\mathcal P_g=\mathcal P_{g_0}$ and  a Taylor expansion using Proposition \ref{prop:deriv}  to obtain the following upper bound on the generalized Poincaré X-ray transform. 
\begin{lemma}\label{lem:tensor-bounds}
    Let $(M, g_0)$ be a hyperbolic manifold. Then there exist constants $\eps > 0$ and  $C, \tau, \alpha > 0$  such that the following holds. For any metric $g \in B_{C^5}(g_0,\epsilon)$ with $\mathcal{P}_g = \mathcal{P}_{g_0}$, we have
    \begin{equation*}\label{eq:111}
\| Q(g-g_0) \|_{C^{\alpha}}   \leq C \|g-g_0\|^{1+\tau}_{C^{5}}.
\end{equation*}
\end{lemma}

\begin{remark}
    Since $\mathcal{P}_g = \mathcal{P}_{c_g \phi^*g}$, Lemma \ref{lem:tensor-bounds} implies
      \begin{equation}\label{eq:SSS}
\| Q( S )\|_{C^{\alpha}}   \leq C \|S\|^{1+\tau}_{C^{5}},
\end{equation}
where $S \in \mathcal{V}_{g_0}$ is as in \eqref{eq:projection}.
\end{remark}

  \begin{proof}
   Using Lemma \ref{lemm:Taylor} (with $k=5$) and Proposition \ref{prop:deriv},  we Taylor expand $\la \mapsto \Phi_{g_0 + \la (g-g_0)}$ near $\la = 0$ to obtain, for all $c \in \mathcal{C}$,
   \begin{equation}
       0=\Phi_{g}-\mathbf{1}=-\frac{1}{4(n-1)\ell_{g_0}(c)}\int_{\gamma_0(c)}\pi_2^*\Delta_L (g-g_0)d\ell_{\gamma_0(c)}+\mathcal O(\|g-g_0\|^2_{C^{5}}),
   \end{equation}
   where the $\mathcal{O}$ is uniform in $c\in \mathcal C.$ 
   This means that
   \begin{equation}
   \label{eq:ineq}
\dfrac{1}{\ell_{g_0}(c)} \int_{\gamma_0(c)} \pi_2^* \Delta_L (g-g_0) d\ell_{\gamma_0(c)} = \mathcal{O}(\|g-g_0\|^2_{C^{5}}), \qquad \forall c \in \mathcal{C}.
   \end{equation}
(Compare with \eqref{eq:livsic} in the corresponding infinitesimal problem.)
Hence, the approximate Liv\v sic theorem of Gou{\"e}zel and Lefeuvre \cite{GL} (see also \cite[Theorem 11.1.5]{Lef}) implies
\begin{equation}
\label{eq:approx}
\pi_2^* \Delta_L (g-g_0) = Xu + h,
\end{equation}
where $u, h \in C^\alpha(SM)$, and $\|h\|_{C^\alpha} \leq C \|\pi_2^*\Delta_L (g-g_0)\|_{C^1}^{1-\tau} \|g-g_0\|^{2\tau}_{C^{5}},$ 
for some positive constants $C,\alpha,\tau > 0$ which are uniform in $S$ and depend only on the geodesic flow of $g_0$. 
Since $\Delta_L$ is a differential operator of order $2$, we have $\Vert \pi_2^* \Delta_L (g-g_0) \Vert_{C^1} \leq C \Vert g-g_0 \Vert_{C^3}$, which then implies the inequality
$
\|h\|_{C^\alpha} \leq C \|g-g_0\|_{C^3}^{1-\tau} \|g-g_0\|^{2\tau}_{C^{5}} \leq C \|g-g_0\|^{1+\tau}_{C^{5}}.
$

Next, applying ${\pi_2}_*\Pi$ to \eqref{eq:approx} and using property \eqref{PiX}, we have 
${\pi_2}_* \Pi\pi_2^*\Delta_L(g-g_0)= {\pi_2}_* \Pi h$. 
Using property \eqref{Pi-bdd} above gives 
\begin{equation*}
\| Q(g-g_0) \|_{C^{\alpha}} = \|{\pi_2}_* \Pi h\|_{C^\alpha} \leq C \|h\|_{C^\alpha} \leq C \|g-g_0\|^{1+\tau}_{C^{5}},
\end{equation*}
which completes the proof.
\end{proof}

We now prove Theorem \ref{theoMain}.
\begin{proof}[Proof of Theorem \ref{theoMain}] Let $(M^n,g_0)$ be a closed hyperbolic manifold such that $\Delta_L$ is injective on TT tensors. By \eqref{prop:DFG}, this is equivalent to $\mathrm{Ker}(\Delta_L)\cap \mathcal V_{g_0}=0$.
 Let $g$ be a metric such that  $\|g-g_0\|_{C^N}<\epsilon$ for a small $\epsilon $ and a large $N$ to be determined later.   We can apply Lemma \ref{coerciveQ} to $S$ defined as in \eqref{eq:projection}. Since $\mathcal{P}_g = \mathcal{P}_{g_0}$, combining this with Lemma \ref{lem:tensor-bounds} applied to $S$ (see \eqref{eq:SSS}), we obtain 
\begin{equation}
\label{eq:important}
\|S\|_{C^{1+\alpha}} \leq C'\|Q(S)\|_{C^\alpha}\leq C'' \|S\|^{1+\tau}_{C^{5}}.
\end{equation}
Interpolating between Hölder spaces, we have
\[
\|S\|^{1+\tau}_{C^{5}} \leq C \|S\|_{C^{1+\alpha}} \|S\|^{\tau}_{C^N},
\]
for  $N=5+\alpha +\tfrac 4 \tau$. Plugging this into \eqref{eq:important} yields
\[
\|S\|_{C^{1+\alpha}} \leq C\|S\|_{C^{1+\alpha}} \|S\|^{\tau}_{C^N}.
\]
Now, suppose that $\|S\|_{C^N}\leq \eps \leq (2C)^{-1/\tau}$. This forces $S \equiv 0$, that is $c_g\phi_g^*g=g_0$, which finishes the proof of the theorem.  
\end{proof}
    Before proving the local rigidity statement of Theorem \ref{theoS}, we first explain how to solve the corresponding \emph{infinitesimal }problem. 
    \begin{proposition}
\label{prop:infinite}
    Let $(g_{\la})_{\la \in (-\eps, \eps)}$ be a 
    deformation of a hyperbolic metric $g_0$ which satisfies $\mathcal{P}_{g_{\la}} \equiv \mathcal{P}_{g_0}$ and $\mathcal S(g_\la)\equiv\mathcal S(g_0).$ 
    Assume in addition that $S = \partial_{\la}|_{\la = 0} g_{\la} \in \mathcal{V}_{g_0}$. 
    Then $S\equiv 0$. 
\end{proposition}
  \begin{proof}
By Proposition \ref{prop:41}, we have $\Delta_L S=0$ and is a TT tensor. Next, we use \cite[Equation (4.2)]{Berger1970} and $\mathrm{Scal}_{g_0}=-n(n-1)$, to obtain for any  $S\in \mathcal V_{g_0}$, 
\begin{equation}
\label{eq:SHess}
\begin{split}d^2_{g_0}\mathcal S(S,S)\mathrm{Vol}_{g_0}&(M)^{(n-2)/n}
\\&=\frac 12\langle \Delta \tr(S),\tr(S)\rangle_{L^2}
+\frac{n-1}2\|\tr(S)\|^2_{L^2}-\frac 12\langle \Delta_L S,S\rangle_{L^2}-(n-1)\|S\|^2_{L^2}.
\end{split}
\end{equation}
Using that $\Delta_L S = 0$ and ${\rm tr}(S) = 0$, the above equation gives 
$$0=\partial_{\la}^2|_{\la=0}\mathcal S(g_\la)=d^2_{g_0}\mathcal S(S,S)=-\mathrm{Vol}_{g_0}(M)^{(2-n)/n}(n-1)\|S\|^2_{L^2},$$
which forces $S\equiv 0$ as desired.
  \end{proof} 

We now make this infinitesimal argument uniform in a neighborhood of $g_0.$
In the setting of the infinitesimal problem, we used the hypothesis $\mathcal{P}_{g_{\la}} \equiv \mathcal{P}_{g_0}$ to conclude that $S \in \ker (\Delta_L)$. 
Here, we instead show, roughly speaking, that if $g$ is close to $g_0$ and $\mathcal P_g=\mathcal P_{g_0}$, then the size of the projection of $S$ onto $\ker (\Delta_L)$ is large relative to that of $S$. 
\begin{lemma}
\label{lemm:1st} There exist $N \in \mathbb{N}$ and $C,\eps > 0$ such that for all $g\in B_{C^N}(g_0,\eps)$ such that $\mathcal P_g=\mathcal P_{g_0}$, the following holds. Let $S$ be as in \eqref{eq:projection} and write $S=T+U$ for $U\in \mathrm{ker}(\Delta_L)$ and $T$ is orthogonal to $ \mathrm{ker}(\Delta_L)$ with respect to the scalar product on $L^2(M;S^2T^*M)$. Then 
 we have $$\|T\|_{C^{1+\alpha}}\leq  C\eps^\tau \|S\|_{L^2}.  $$
\end{lemma}

\begin{proof}
We first note that since $S \in \mathcal{V}_{g_0}$, we have $T, U \in \mathcal{V}_{g_0}$. Indeed, $\Delta_L$ preserves $\mathcal V_{g_0}$ (Lemma \ref{lemm:DFG}) and is self-adjoint, hence its spectral projector onto $\mathrm{Ker}(\Delta_L)$ preserves $\mathcal V_{g_0}.$  

Since we do not know that $\Delta_L$ is injective on $\mathcal{V}_{g_0}$, we cannot directly apply Lemma \ref{coerciveQ} to the tensor $S$ as in the proof of Theorem \ref{theoMain}.
However, we have $\Pi_{\mathrm{ker}(\Delta_L)}(T)=0$ by construction, so applying 
the estimate to $T$  gives
$$ \|T\|_{C^{1+\alpha}} \leq C\|Q(T)\|_{C^\alpha}=C\|Q(S)\|_{C^{\alpha}},$$
where for the second equality we used that $U\in \mathrm{Ker}(\Delta_L)$, see \eqref{eq:Q}. 
Combining this with Lemma \ref{lem:tensor-bounds} applied to $S$ (see \eqref{eq:SSS}), we deduce
$$\|T\|_{C^{1+\alpha}} \leq C\|S\|_{C^{5}}^{1+\tau}. $$
The right-hand side can be bounded using a Sobolev embedding and interpolating between Sobolev spaces. For any $\beta>0,$ we have
$$ \|S\|_{C^{5}}^{1+\tau}\leq c\|S\|_{H^{n/2+5+\beta}}^{1+\tau}\leq c'\|S\|_{H^N}^\tau\|S\|_{L^2}\leq c''\|S\|_{C^N}^\tau\|S\|_{L^2},$$
for some constants $c,c'$ and $c''$ independent of $g$ and where $N=\tfrac{1+\tau}{\tau}(n/2+5+\beta).$ If we plug this last estimate into the upper bound for $T$, we obtain
$$ \|T\|_{C^{1+\alpha}}\leq c'\|S\|_{C^N}^\tau\|S\|_{L^2}. $$
To finish the proof, we recall that by Lemma \ref{slice lemma},  if $\|g-g_0\|_{C^N}<\eps$ for $\eps>0$, small enough, then 
$\|S\|_{C^N}\leq C\eps$, from which the desired result follows.
\end{proof}
The next step consists in controlling the Hessian of $\mathcal S$.
\begin{lemma}
\label{lemm:2} 
   Let $g$, $S$, $\epsilon$ and $N$ be as in Lemma \ref{lemm:1st}. 
   For all $g\in B_{C^N}(g_0,\eps)$, we have
    $$d_{g_0}^2\mathcal S(S,S)= (-(n-1)\mathrm{Vol}_{g_0}(M)^{(2-n)/n}+\mathcal O(\eps^{2\tau}))\|S\|_{L^2}^2. $$
\end{lemma}
\begin{proof}
    The polarization identity applied to \eqref{eq:SHess} gives
    $d^2_{g_0}\mathcal S(T,U)=0,$
    which implies that
$$d_{g_0}^2\mathcal S(S,S)=d_{g_0}^2\mathcal S(T,T)+d_{g_0}^2\mathcal S(U,U). $$
Using the Hessian formula \eqref{eq:SHess} and integration by parts we obtain 
$$|d_{g_0}^2\mathcal S(T,T)|\leq C\|T\|_{H^1}^2 \leq C' \Vert T \Vert^2_{C^{1+\alpha}}, $$
for some constants $C,C'>0$ independent of $g \in B_{C^N}(g_0,\eps).$
Next, using that $U\in \mathrm{Ker}(\Delta_L)$ in \eqref{eq:SHess} gives
$$d_{g_0}^2\mathcal S(U,U)=-(n-1)\mathrm{Vol}_{g_0}(M)^{(2-n)/n}\|U\|_{L^2}^2.$$
 The lemma now follows from Lemma \ref{lemm:1st} and the fact that $\|S\|_{L^2}^2=\|T\|_{L^2}^2+\|U\|_{L^2}^2$.
\end{proof}
We now prove Theorem \ref{theoS}. 
\begin{proof}[Proof of Theorem \ref{theoS}] Let $g$ be a metric such that  $\|g-g_0\|_{C^N}<\epsilon$ for a small $\epsilon $ and a large $N$ to be determined later. Suppose that $\mathcal P_g=\mathcal P_{g_0}$ and $\mathcal S(g)=\mathcal S(g_0).$ Let $S$ be as in \eqref{eq:projection}.
Since the mapping
$\mathcal S:\mathrm{Met}^2(M)\to \mathbb R $
is $C^{\infty}$ (here $\mathrm{Met}^2(M)$ denotes the set of $C^2$ metrics on $M$), we can consider the 
Taylor expansion of $\lambda \mapsto \mathcal S(g_0 + \lambda S)$ at $\lambda = 0$. 
Since $\mathcal S$ is invariant by homotheties and has a critical point at the hyperbolic metric $g_0$, we have
$$\mathcal{S}(g) = \mathcal S(c_g\phi_g^* g)=\mathcal S( g_0+S)=\mathcal S(g_0)+\frac 12 d^2_{g_0}\mathcal S(S,S)+\mathcal{O}(\|S\|_{C^2}^3). $$

Sine $\mathcal{S}(g) = \mathcal{S}(g_0)$, Lemma \ref{lemm:2} gives
there is $\eps>0$ small enough so that
$\|S\|_{L^2}^2\leq C\|S\|_{C^2}^3, $
for some constant $C>0$ which does not depend on $g$. Using a Sobolev embedding and interpolation, we deduce that for any $\beta>0,$ 
$$\|S\|_{C^{2}}^3\leq c\|S\|_{H^{n/2+2+\beta}}^3\leq c'\|S\|_{L^2}^2\|S\|_{H^N}\leq c'\|S\|_{L^2}^2\|S\|_{C^N}, $$
for any $N>\tfrac 3 2 n+6$ and for some constants $c,c'$ independent of $g$. In total, this implies
$$\|S\|_{L^2}^2\leq Cc'\|S\|_{L^2}^2\|S\|_{C^N}. $$
If $\|S\|_{C^N}<\epsilon<1/(Cc')$, this implies that $S\equiv 0$, that is, $c_g\phi_g^* g=g_0$. This concludes the proof of Theorem \ref{theoS}.
\end{proof}
We now finish with the proof of Theorem \ref{theo-conf}. Note the proof is very similar to the proof of Theorem \ref{theoMain} (where we suppose $\Delta_L$ is injective) because $\Delta_L$ is injective on conformal tensors of zero mean. 
\begin{proof}[Proof of Theorem \ref{theo-conf}]
    Let $g\in B_{C^N}(g_0,\eps)$ be a metric conformal to $g_0$ and let $c>0$ be such that $S:=cg-g_0=fg_0$ has zero mean, that is, $\int_M fd\mathrm{vol}_{g_0}=0.$
    By \eqref{eq:LaplFla}, we have $\Delta_L S=(\Delta f)g_0$. Suppose that $\Pi_{\mathrm{Ker}(D_{g_0}^*)}\Delta_L(fg_0)=0,$ then we claim that $\Delta_L(fg_0)=0$. Indeed, this follows from the injectivity of the X-ray transform on $0$-tensors \cite{DS10}.

    Combined with the standard elliptic estimate for the Laplacian $\Delta$, this gives
    $$ \|f\|_{C^{\alpha+1}}\leq C\|S\|^{1+\tau}_{C^5}.$$ The same interpolation argument as in the proof of Theorem \ref{theoMain} shows that $f=0$ if we have $g\in B_{C^N}(g_0,\eps)$  for $N$ large enough and $\eps$ small enough.
\end{proof}

\section{Examples of hyperbolic $3$-manifolds with injective Lichnerowicz Laplacian on TT tensors} 
\label{sec:examples}
In this section, we explain how the work of Bonifacio, Mazáč, and Pal \cite{BMP25} provides strong numerical evidence of the existence of hyperbolic $3$-manifolds with injective Lichnerowicz Laplacian on TT tensors.

\begin{definition}
   In dimension $n=3$, we define the \emph{curl operator} 
\begin{align*}
    \mathrm{curl}:C^\infty (M;S^mT^*M) &\to C^\infty (M;S^mT^*M) \\
   (X_1,\ldots,X_m) &\mapsto \frac 1m \sum_{i=1}^m\nabla S(*X_i,X_1,\ldots,\hat{X_i},\ldots,X_m), 
\end{align*}   
where $*X$ denotes the Hodge star of $X$, see \cite[Definition 2.19]{BMP25} for the definition. 
\end{definition}

\begin{lemma}\label{lemma: injectivity of Lichnerowicz from the spectrum} Let $(M^3,g_0)$ be a closed  hyperbolic $3$-manifold. Then the Lichnerowicz Laplacian of $g_0$ is injective on TT tensors if and only if $\pm \sqrt{3}$ are \emph{not} in the spectrum of the curl operator acting on TT tensors.\end{lemma}
\begin{proof}
By \eqref{eq:LaplFla}, we have $\Delta_L$ is injective on TT tensors if and only if $6 \notin {\rm Spec} (\nabla^* \nabla)$ for the action on TT tensors. Next, we have 
\begin{equation}
    \label{eq:crulfact}
  \forall S\in \mathrm{ker}(\tr)\cap \mathrm{ker}(D_{g_0}^*),\quad \nabla^*\nabla S=( \mathrm{curl}\circ \mathrm{curl}+3)S,
\end{equation}
(see \cite[page 3]{BMP25} for the  equality).
Using the self-adjointness of the curl operator concludes the proof.
\end{proof}

In \cite[Section 5.3]{BMP25}, the authors outline a strategy for ruling out certain eigenvalues $t_i^{(J)}$ of the curl operator acting on symmetric trace-free divergence-free $J$-tensors using the Selberg trace formula. More precisely, they explain how applying \cite[Equation (5.4)]{BMP25} to well-chosen test functions $H$ produces a disjoint union of intervals containing the absolute values of the curl operator eigenvalues; see \cite[Equation (5.9)]{BMP25}. 

As explained in \cite[Section 5.6]{BMP25}, they apply their above procedure to various examples of hyperbolic 3-manifolds (of small volume).
The resulting intervals can be found in the ancillary files \cite{BMP25aux}.

In order to evaluate the geometric side of the trace formula, the authors use SnapPy \cite{CDGW_SnapPy} to compute the length spectra of their examples using fixed precision floating-point numbers. 
As explained to us by James Bonifacio, the computations are not certified using interval arithmetic, though in practice, the errors are expected to be very small.
However, in order to rigorously exclude eigenvalues of the curl operator, one would have to verify how the length errors propagate through the trace formula.

Inspection of the ancillary files \cite{BMP25aux} shows that there are 57 examples hyperbolic $3$-manifolds for which their numerical computations suggest that the Lichnerowicz Laplacian is injective on TT tensors.
For instance, for the Weeks manifold, denoted by $m003(-3,1)$ in SnapPy notation, the file \cite{BMP25aux} states that any eigenvalue $t_j^{(2)}$ of the curl operator acting on TT symmetric $2$-tensors satisfies
\begin{align*}
|t_j^{(2)}|&\in [2.1489828527, 2.1523921871]\cup [3.0587150013, 3.0723408285]
\\&\cup[4.5128836759, 4.6685522999]\cup [4.9056862981,\infty).
\end{align*}
Since $\sqrt 3\approx 1.73205080757<2.1489828527$, this strongly suggests that $\Delta_L$ is injective on TT tensors for the Weeks manifold.

In SnapPy notation, the other examples are given by: $m$003(-2,3), $m$003(-3,1), $m$003(-3,4), $m$003(-4,3), $m$003(-5,3), $m$003(-5,4), $m$004(6,1), $m$004(7,1), $m$004(8,1), $m$006(1,3), $m$007(1,2), $m$007(3,1), $m$007(3,2), $m$007(5,1), $m$009(3,2), $m$009(4,1), $m$009(-5,1), $m$009(5,1), $m$009(6,1), $m$009(7,1), $m$010(-2,3), $m$010(-3,4), $m$010(4,1), $m$010(-4,3), $m$010(5,1), $m$010(6,1), $m$011(2,3), $m$015(5,1), $m$015(6,1), $m$017(-1,3), $m$019(3,2), $m$023(-3,2), $m$023(-4,1), $m$023(-5,1), $m$023(-6,1), $m$030(6,1), $m$030(7,1), $m$032(5,1), $m$034(-4,1), $m$038(3,2), $m$038(4,1), $m$038(5,1), $m$082(1,3), $m$120(-4,1), $m$120(-5,1), $m$135(-1,3), $m$135(1,3), $m$140(-4,1), $m$140(-5,1), $m$160(-3,2), $m$168(2,3), $m$207(-1,3), $m$286(-4,1), $s$090(5,1), $s$118(-4,1), $s$254(-3,1), $s$254(-4,1).

\bibliography{ref}{}

@article{GRH-abelian,
  title={Abelian Livshits theorems and geometric applications},
  author={Gogolev, Andrey and Rodriguez Hertz, Federico},
  journal={A vision for dynamics in the 21st century—the legacy of Anatole Katok},
  pages={139--167},
  year={2024}
}

@article{grh,
  title={Smooth rigidity for 3-dimensional volume preserving {A}nosov flows and weighted marked length spectrum rigidity},
  author={Gogolev, Andrey and Rodriguez Hertz, Federico},
  journal={arXiv preprint arXiv:2210.02295},
  year={2022}
}

@article{gromov,
  title={Three remarks on geodesic dynamics and fundamental group},
  author={Gromov, Mikha{\i}l},
  journal={Enseign. Math.(2)},
  volume={46},
  pages={391--402},
  year={2000}
}

@article{dL,
  title={Smooth conjugacy and {SRB} measures for uniformly and non-uniformly hyperbolic systems},
  author={de la Llave, Rafael},
  journal={Communications in mathematical physics},
  volume={150},
  number={2},
  pages={289--320},
  year={1992},
  publisher={Springer}
}

@article{butler1,
  title={Rigidity of equality of {L}yapunov exponents for geodesic flows},
  author={Butler, Clark},
  journal={Journal of Differential Geometry},
  volume={109},
  number={1},
  pages={39--79},
  year={2018},
  publisher={Lehigh University}
}

@article{butler2,
  title={Characterizing symmetric spaces by their {L}yapunov spectra},
  author={Butler, Clark},
  journal={arXiv:1709.08066v6},
  year={2017}
}

@article{BCG,
  title={Entropies et rigidit{\'e}s des espaces localement sym{\'e}triques de courbure strictement n{\'e}gative},
  author={Besson, G{\'e}rard and Courtois, Gilles and Gallot, Sylvestre},
  journal={Geometric \& Functional Analysis GAFA},
  volume={5},
  number={5},
  pages={731--799},
  year={1995},
  publisher={Birkh{\"a}user-Verlag}
}

@article{ka-conf,
  title={Four applications of conformal equivalence to geometry and dynamics},
  author={Katok, Anatole},
  journal={Ergodic Theory Dynam. Systems},
  volume={8},
  number={Charles Conley Memorial Issue},
  pages={139--152},
  year={1988}
}

@article{otal,
  title={Le spectre marqu{\'e} des longueurs des surfaces {\`a} courbure n{\'e}gative},
  author={Otal, Jean-Pierre},
  journal={Annals of Mathematics},
  volume={131},
  number={1},
  pages={151--162},
  year={1990},
  publisher={JSTOR}
}

@article{croke,
  title={Rigidity for surfaces of non-positive curvature},
  author={Croke, Christopher B},
  journal={Commentarii Mathematici Helvetici},
  volume={65},
  number={1},
  pages={150--169},
  year={1990},
  publisher={Springer}
}

@article{CS98,
  author  = {Croke, Christopher B. and Sharafutdinov, Vladimir A.},
  title   = {Spectral rigidity of a compact negatively curved manifold},
  journal = {Topology},
  volume  = {37},
  number  = {6},
  pages   = {1265--1273},
  year    = {1998},
  doi     = {10.1016/S0040-9383(97)00086-4}
}

@article{GLP,
  title={Marked length spectrum rigidity for {A}nosov surfaces},
  author={Guillarmou, Colin and Lefeuvre, Thibault and Paternain, Gabriel P},
  journal={Duke Mathematical Journal},
  volume={174},
  number={1},
  pages={131--157},
  year={2025},
  publisher={Duke University Press}
}

@article{ham99,
  title={Cocycles, symplectic structures and intersection},
  author={Hamenst{\"a}dt, Ursula},
  journal={Geometric \& Functional Analysis GAFA},
  volume={9},
  number={1},
  pages={90--140},
  year={1999},
  publisher={Springer}
}

@article{GK80,
  title={Some inverse spectral results for negatively curved 2-manifolds},
  author={Guillemin, Victor and Kazhdan, David},
  journal={Topology},
  volume={19},
  number={3},
  pages={301--312},
  year={1980},
  publisher={Pergamon}
}

@book{Bal,
  title={Lectures on spaces of nonpositive curvature},
  author={Ballmann, Werner},
  volume={25},
  year={1995},
  publisher={Springer Science \& Business Media}
}

@article{burnskatok,
  title={Manifolds with non-positive curvature},
  author={Burns, Keith and Katok, Anatole and Ballman, Werner and Brin, Michael and Eberlein, Patrick and Osserman, Robert},
  journal={Ergodic Theory and Dynamical Systems},
  volume={5},
  number={2},
  pages={307--317},
  year={1985},
  publisher={Cambridge University Press}
}

@article{Con,
author = {Contreras, Gonzalo},
issn = {0025-5874},
journal = {Mathematische Zeitschrift},
language = {eng},
number = {1},
pages = {97-111},
publisher = {Springer Science and Business Media LLC},
title = {Regularity of topological and metric entropy of hyperbolic flows},
volume = {210},
year = {1992},
}

@article{GL19,
  title={The marked length spectrum of {A}nosov manifolds},
  author={Guillarmou, Colin and Lefeuvre, Thibault},
  journal={Annals of Mathematics},
  volume={190},
  number={1},
  pages={321--344},
  year={2019},
  publisher={JSTOR}
}

@article{Fla,
  TITLE = {{Local entropy rigidity for hyperbolic manifolds}},
  AUTHOR = {Flaminio, Livio},
  URL = {https://hal.science/hal-04211487},
  JOURNAL = {{Communications in Analysis and Geometry}},
  PUBLISHER = {{International Press}},
  VOLUME = {3},
  NUMBER = {4},
  PAGES = {555-596},
  YEAR = {1995},
  DOI = {10.4310/CAG.1995.v3.n4.a2},
  HAL_ID = {hal-04211487},
  HAL_VERSION = {v1},
}

@article{GuKnLef, title={Geodesic stretch, pressure metric and marked length spectrum rigidity}, volume={42}, DOI={10.1017/etds.2021.75}, number={3}, journal={Ergodic Theory and Dynamical Systems}, author={Guillarmou, Colin and Knieper, Gerhard and Lefeuvre, Thibault}, year={2022}, pages={974–1022}}

@book{Lef,  author ={Lefeuvre, Thibault},
    title = {Microlocal analysis in hyperbolic dynamics and geometry},
     publisher={\href{https://thibaultlefeuvre.blog/wp-content/uploads/2025/03/microlocal-analysis-in-hyperbolic-dynamics-and-geometry-1.pdf}{thibaultlefeuvre.blog/microlocal-analysis-in-hyperbolic-dynamics-and-geometry-2}},
     year={2025} 
}

@article{Hum,
  author  = {Humbert, Tristan},
  title   = {Katok's entropy conjecture near real and complex hyperbolic metrics},
  journal = {Duke Mathematical Journal},
  volume  = {175},
  number  = {1},
  pages   = {135--184},
  year    = {2026},
  doi     = {10.1215/00127094-2025-0023}
}

@misc{CDGW_SnapPy,
  author       = {Culler, Marc and Dunfield, Nathan M. and
                  Goerner, Matthias and Weeks, Jeffrey R.},
  title        = {Snap{P}y, a computer program for studying the
                  geometry and topology of $3$-manifolds},
  howpublished = {Available at \url{https://snappy.computop.org/}}
}

@article{GL,
author = {S{\'e}bastien Gou{\"e}zel and Thibault Lefeuvre},
title = {{Classical and microlocal analysis of the {X}-ray transform on {A}nosov manifolds}},
volume = {14},
journal = {Analysis \& PDE},
number = {1},
publisher = {MSP},
pages = {301 -- 322},
year = {2021},
doi = {10.2140/apde.2021.14.301},
URL = {https://doi.org/10.2140/apde.2021.14.301}
}

@article{GBL,
  TITLE = {{Radial source estimates in H{\"o}lder-Zygmund spaces for hyperbolic dynamics}},
  AUTHOR = {Guedes Bonthonneau, Yannick and Lefeuvre, Thibault},
  URL = {https://hal.science/hal-03435760},
  NOTE = {46 pages, 1 figure},
  JOURNAL = {{Annales Henri Lebesgue}},
  PUBLISHER = {{UFR de Math{\'e}matiques - IRMAR}},
  VOLUME = {6},
  PAGES = {643-686},
  YEAR = {2023},
  MONTH = Oct,
  DOI = {10.5802/ahl.175},
  HAL_ID = {hal-03435760},
  HAL_VERSION = {v1},
}

@article{DFG,
author = {Dyatlov, Semyon and Faure, Frédéric and Guillarmou, Colin},
copyright = {Distributed under a Creative Commons Attribution 4.0 International License},
issn = {2157-5045},
journal = {Analysis and PDE},
language = {eng},
number = {4},
pages = {923-1000},
publisher = {Mathematical Sciences Publishers},
title = {Power spectrum of the geodesic flow on hyperbolic manifolds},
volume = {8},
year = {2015},
}

@book{Bes,
author = {Arthur L. Besse},
booktitle = {Einstein manifolds},
isbn = {0-387-15279-2},
publisher = {Springer},
series = {Ergebnisse der Mathematik und ihrer Grenzgebiete. Folge 3},
title = {Einstein manifolds},
year = {1987},
}

@article{Kn, title={A second derivative formula of the {L}iouville entropy at spaces of constant negative curvature}, volume={17}, DOI={10.1017/S0143385797086446}, number={5}, journal={Ergodic Theory and Dynamical Systems}, author={Knieper, Gerhard}, year={1997}, pages={1131–1135}}

@article{Maubon,
  title={Variations d'entropies et d{\'e}formations de structures conformes plates sur les vari{\'e}t{\'e}s hyperboliques compactes},
  author={Maubon, Julien},
  journal={Ergodic Theory and Dynamical Systems},
  volume={20},
  number={6},
  pages={1735--1748},
  year={2000},
  publisher={Cambridge University Press}
}

@article{Ka,
title = "Entropy and closed geodesics",
author = "Anatole Katok",
year = "1982",
volume = "2",
pages = "339--365",
journal = "Ergodic Theory and Dynamical Systems",
issn = "0143-3857",
publisher = "Cambridge University Press",
number = "3-4",
}

@article{Eb,
author = {Ebin, David G.},
issn = {0273-0979},
journal = {Bulletin (new series) of the American Mathematical Society},
language = {eng},
number = {5},
pages = {1001-1003},
title = {On the space of {R}iemannian metrics},
volume = {74},
year = {1968},
}

@article{Gu,
  TITLE = {{Invariant distributions and {X}-ray transform for {A}nosov flows}},
  AUTHOR = {Guillarmou, Colin},
  URL = {https://hal.science/hal-01236227},
  JOURNAL = {{Journal of Differential Geometry}},
  PUBLISHER = {{International Press}},
  VOLUME = {105},
  NUMBER = {2},
  PAGES = {177--208},
  YEAR = {2017},
  MONTH = Feb,
  HAL_ID = {hal-01236227},
  HAL_VERSION = {v1},
}

@article{Liv,
author = {Liverani, Carlangelo},
address = {Princeton, NJ},
copyright = {Copyright 2004 Princeton University (Mathematics Department)},
issn = {0003-486X},
journal = {Annals of mathematics},
language = {eng},
number = {3},
pages = {1275-1312},
publisher = {Princeton University Press},
year={2004},
title = {On Contact {A}nosov Flows},
volume = {159},
}

@article{DS10,
  author  = {Dairbekov, Nurlan S. and Sharafutdinov, Vladimir A.},
  title   = {On conformal {K}illing symmetric tensor fields on {R}iemannian manifolds},
  journal = {Mat. Tr.},
  volume  = {13},
  number  = {1},
  pages   = {85--145},
  year    = {2010}
}

@article{Sch,
    author = {Schwahn, Paul},
    title = {Stability of {E}instein metrics on symmetric spaces of compact type},
    journal = {Annals of Global Analysis and Geometry},
    year = {2022},
    volume = {61},
    number = {2},
    pages = {333-357}
}

@article{RST,
  title={Geometry in the large of the kernel of {L}ichnerowicz {L}aplacians and its applications},
  author={Vladimir Yu. Rovenski and Sergey E. Stepanov and Irina I. Tsyganok},
  journal={Balkan Journal of Geometry and its Applications},
  volume={25},
  number={2},
  year={2020}
}

@Article{MST,
AUTHOR = {Mikeš, Josef and Stepanov, Sergey and Tsyganok, Irina},
TITLE = {The {L}ichnerowicz-Type {L}aplacians: Vanishing Theorems for Their Kernels and Estimate Theorems for Their Smallest Eigenvalues},
JOURNAL = {Mathematics},
VOLUME = {12},
YEAR = {2024},
NUMBER = {24},
ARTICLE-NUMBER = {3936},
URL = {https://www.mdpi.com/2227-7390/12/24/3936},
ISSN = {2227-7390},
DOI = {10.3390/math12243936}
}

@article{Del,
title = {Essential spectrum of the {L}ichnerowicz {L}aplacian on two tensors on asymptotically hyperbolic manifolds},
journal = {Journal of Geometry and Physics},
volume = {43},
number = {1},
pages = {33-44},
year = {2002},
issn = {0393-0440},
doi = {https://doi.org/10.1016/S0393-0440(02)00007-4},
url = {https://www.sciencedirect.com/science/article/pii/S0393044002000074},
author = {Erwann Delay}
}

@article{BMP25,
  author  = {Bonifacio, James and Maz{\'a}{\v{c}}, Dalimil and Pal, Sridip},
  title   = {Spectral {B}ounds on {H}yperbolic 3-{M}anifolds:
             {A}ssociativity and the {T}race {F}ormula},
  journal = {Communications in Mathematical Physics},
  volume  = {406},
  number  = {3},
  pages   = {51},
  year    = {2025},
  doi     = {10.1007/s00220-024-05222-0},
  eprint  = {2308.11174},
  archivePrefix = {arXiv}
}

@article{BMP25aux,
  author       = {Bonifacio, James and Maz{\'a}{\v c}, Dalimil and Pal, Sridip},
  title        = {Ancillary files for ``{S}pectral {B}ounds on {H}yperbolic 3-{M}anifolds:
             {A}ssociativity and the {T}race {F}ormula''},
journal={arXiv: 2308.11174},
  year         = {2023},
  note         = {File \texttt{3D-spectral-data.txt}},
  url          = {https://arxiv.org/src/2308.11174v1/anc/3D-spectral-data.txt}
}

@article{KKPW,
author ={Anatole Katok and Gerhard Knieper and Mark Pollicott and Howard Weiss},
issn = {0020-9910},
journal = {Inventiones mathematicae},
language = {eng},
number = {3},
pages = {581-597},
publisher = {Springer Science and Business Media LLC},
title = {Differentiability and analyticity of topological entropy for {A}nosov and geodesic flows},
volume = {98},
year = {1989},
}

@article{Berger1970,
  author  = {Berger, Marcel},
  title   = {Quelques formules de variation pour une structure riemannienne},
  journal = {Annales scientifiques de l'\'Ecole Normale Sup\'erieure},
  series  = {4},
  volume  = {3},
  number  = {3},
  year    = {1970},
  pages   = {285--294},
  doi     = {10.24033/asens.1194}
}

@article{GLRH25,
      title={Smooth rigidity for 3-dimensional dissipative Anosov flows}, 
      author={Andrey Gogolev and Martin Leguil and Federico Rodriguez Hertz},
      year={2025},
      journal={arXiv: 2510.23872}
}
\bibliographystyle{alpha}

\end{document}